\documentclass[12pt]{amsart}
\usepackage{amscd,amssymb}
\usepackage[arrow,matrix]{xy}
\usepackage[colorlinks,plainpages,backref,urlcolor=blue]{hyperref}

\newtheorem{theorem}[subsection]{Theorem}
\newtheorem{lemma}[subsection]{Lemma}
\newtheorem{prop}[subsection]{Proposition}
\newtheorem{corollary}[subsection]{Corollary}

\newtheorem{thm}{Theorem}

\theoremstyle{definition}
\newtheorem{remark}[subsection]{Remark}

\newtheorem{example}[subsection]{Example}
\newtheorem*{ack}{Acknowledgments}

\numberwithin{equation}{section}
\newcommand{\cH}{{\mathfrak h}}  
\newcommand{\cK}{{\mathfrak k}} 
\newcommand{\ccG}{{\mathcal G}}
\newcommand{\ccB}{{\mathcal B}}

\newcommand{\M}{{\mathfrak g}}  

\newcommand{\V}{{\mathcal V}}
\newcommand{\R}{{\mathcal R}}
\newcommand{\TT}{{\mathcal T}}

\newcommand{\cG}{\mathfrak {g}}
\newcommand{\fp}{\mathfrak {p}}
\newcommand{\fq}{\mathfrak {q}}

\newcommand{\bb}{\mathfrak {b}}

\renewcommand{\sp}{\mathfrak {sp}}
\renewcommand{\sl}{\mathfrak {sl}}

\newcommand{\oG}{\overline{G}}

\newcommand{\kk}{\kappa}

\newcommand{\T}{\mathbb{T}}
\newcommand{\Z}{\mathbb{Z}}

\newcommand{\Q}{\mathbb{Q}}

\newcommand{\C}{\mathbb{C}}

\newcommand{\bL}{\mathbb{L}}

\newcommand{\G}{\Gamma}

\def\Sp{\mathrm{Sp}}
\def\dot{\bullet}

\DeclareMathOperator{\Spec}{Specm}
\DeclareMathOperator{\spec}{Spec}
\DeclareMathOperator{\Hom}{Hom}
\DeclareMathOperator{\rank}{rank}
\DeclareMathOperator{\im}{im}

\DeclareMathOperator{\coker}{coker}

\DeclareMathOperator{\id}{id}

\DeclareMathOperator{\supp}{supp}

\DeclareMathOperator{\gr}{gr}

\DeclareMathOperator{\ab}{ab}
\DeclareMathOperator{\ann}{ann}
\DeclareMathOperator{\abf}{abf}
\DeclareMathOperator{\ad}{ad}

\DeclareMathOperator{\Sym}{Sym}

\DeclareMathOperator{\Aut}{Aut}

\newcommand{\surj}{\twoheadrightarrow}
\newcommand{\inj}{\hookrightarrow}
\newcommand{\isom}{\xrightarrow{\,\simeq\,}}
\newcommand{\eqv}{{\Longleftrightarrow}}

\date{May 4, 2012}

\begin{document}

\title [The abelianization of the Johnson kernel]
{The abelianization of the Johnson kernel}

\author[Alexandru Dimca]{Alexandru Dimca$^1$}
\address{Institut Universitaire de France et Laboratoire J.A. Dieudonn\'e, UMR
                              du CNRS 7351,
                 Universit\'e de Nice Sophia-Antipolis,
                 Parc Valrose,
                 06108 Nice Cedex 02,
                 France}
\email{dimca@unice.fr}

\author[Richard Hain]{Richard Hain$^2$}
\address{Department of Mathematics, Duke University, Durham,
NC 27708-0320, USA}
\email{hain@math.duke.edu}

\author[Stefan Papadima]{Stefan Papadima$^3$}
\address{Simion Stoilow Institute of Mathematics, 
P.O. Box 1-764,
RO-014700 Bucharest, Romania}
\email{Stefan.Papadima@imar.ro}

\thanks{$^1$ Partially supported by ANR-08-BLAN-0317-02 (SEDIGA)} 
\thanks{$^2$ Supported in part by grant DMS-1005675 from the National Science
Foundation.}
\thanks{$^3$ Partially supported by CNCSIS-UEFISCSU project PNII-IDEI 1189/2008}



\begin{abstract}
We prove that the first complex homology of the Johnson subgroup of the Torelli
group $T_g$ is a non-trivial, unipotent $T_g$-module for all $g\ge 4$ and give
an explicit presentation of it as a $\Sym_\dot H_1(T_g,\C)$-module when $g\ge
6$. We do this by proving that, for a finitely generated group $G$ satisfying an
assumption close to formality, the triviality of the restricted characteristic
variety implies that the first homology of its Johnson kernel $K$ is a nilpotent
module over the corresponding Laurent polynomial ring, isomorphic to the
infinitesimal Alexander invariant of the associated graded Lie algebra of $G$.
In this setup, we also obtain a precise nilpotence test.
\end{abstract}

\maketitle

\tableofcontents

\section{Introduction} \label{sec:intro}

Fix a closed oriented surface $\Sigma$ of genus $g \ge 2$. The genus $g$ mapping
class group $\G_g$ is defined to be the group of isotopy classes of orientation
preserving diffeomorphisms of $\Sigma$. For a commutative ring $R$, denote
$H_1(\Sigma,R)$ by $H_R$. The intersection pairing $\theta : H_R\otimes H_R \to
R$ is a unimodular, skew-symmetric bilinear form. Set $\Sp(H_R) =
\Aut(H_R,\theta)$. The action of $\G_g$ on $\Sigma$ induces a surjective
homomorphism $r : \G_g \to \Sp(H_\Z)$. 
The Torelli group $T_g$ is defined to be the kernel of $r$. One thus has
the extension
\begin{equation}
\label{eq:deft}
\xymatrix{
1 \ar[r] & T_g \ar[r] & \G_g \ar[r]^(.4)r & \Sp(H_\Z) \ar[r] & 1.
}
\end{equation}
Dennis Johnson \cite{J1} proved that $T_g$ is finitely generated when $g\ge 3$.

The intersection form $\theta$ spans a copy of the trivial representation in 
$\wedge^2 H_R$. One therefore has the
$\Sp(H_R)$-module
$$
V_R := (\wedge^3 H_R)/(\theta\wedge H_R)
$$
which is torsion free as an $R$-module for all $R$. 

Johnson \cite{J0} constructed a surjective morphism (the ``Johnson
homomorphism'') $\tau : T_g \to V_\Z$ 
and proved in \cite{J3} that it induces an $\Sp(H_\Z)$-module isomorphism
$$
\bar{\tau}: H_1(T_g)/(2\text{-torsion}) \to V_\Z \, .
$$
The {\em Johnson group} $K_g$ is the kernel of $\tau$. By a fundamental result
of Johnson \cite{J2}, it is the subgroup of $\G_g$ generated by Dehn twists on
separating simple closed curves.

The goal of this paper is to describe the $\G_g/K_g$-module $H_1(K_g,\C)$. The
first and third authors \cite{DP} proved that $H_1(K_g,\C)$ is finite
dimensional whenever $g\ge 4$. Our first result is:

\begin{thm}
\label{thm:main1}
If $g\ge 4$, then $H_1(K_g,\C)$ is a non-trivial, unipotent $H_1(T_g)$-module
and $H_1(T_g,\C_\rho)$ vanishes for all non-trivial characters $\rho$ in the
identity component $\Hom_\Z(V_\Z,\C^\ast)$ of $H^1(T_g,\C^\ast)$.
\end{thm}

When $g\ge 6$ we find a presentation of $H_1(K_g,\C)$ as a $\G_g/K_g$-module.
Describing this module structure requires some preparation.

Suppose that $g\ge 3$. Denote the highest weight summand of the second symmetric
power of  the $\Sp(H_\C)$-module $\wedge^2 H_\C$ by $Q$.\footnote{If
$\lambda_1,\dots,\lambda_g$ is a set of fundamental weights of $\Sp(H_\C)$, then
$Q$ is the irreducible module with highest weight $2\lambda_2$. Alternatively,
it is the irreducible module corresponding to the partition $[2,2]$.} There is a
unique $\Sp(H_\C)$-module projection (up to multiplication by a non-zero
scalar), $\pi : \wedge^2 V_\C \surj Q$.

Define a left $\Sym_\dot(V_\C)$-module homomorphism
$$
q: \Sym_\dot(V_\C)\otimes\wedge^3 V_\C \to \Sym_\dot(V_\C)\otimes Q
$$
by
$$
q(f\otimes (a_0\wedge a_1 \wedge a_2)) = \sum_{i\in \Z/3}
f\cdot a_i\otimes \pi(a_{i+1}\wedge a_{i+2}) \, .
$$
The map $q$ is $\Sp(H_\C)$-equivariant. Thus, the cokernel of $q$ is both a
$\Sp(H_\C)$-module and a graded $\Sym_\dot(V_\C)$-module. We show in Section
\ref{sec:infalex} that $\coker (q)$ is finite-dimensional, when $g\ge 6$. It
follows that $\coker (q)$ is a $(\Sp(H_\C)\ltimes V_\C)$-module, where $v\in
V_\C$ acts via its exponential $\exp v$. One therefore has the
$(\Sp(H_\C)\ltimes V_\C)$-module
\begin{equation}
\label{eq:presintro}
M := \C \oplus \coker (q) \, ,
\end{equation}
where $\C$ denotes the trivial module.

To relate the $(\Sp(H_\C)\ltimes V_\C)$-action on $M$ to the $\G_g/K_g$-action
on $H_1(K_g, \C)$, we recall that Morita \cite{Mo2} has shown that there is a
Zariski dense embedding $\G_g/K_g \hookrightarrow \Sp(H_\C)\ltimes V_\C$, unique
up to conjugation by an element of $V_\C$, such that the diagram
$$
\xymatrix{
1 \ar[r] & T_g/K_g \ar[r]\ar[d]_{\bar{\tau}} & \G_g/K_g \ar[r]\ar[d] &
\Sp(H_\Z)\ar[d] \ar[r] & 1 \cr 
1 \ar[r] & V_\C \ar[r] & \Sp(H_\C)\ltimes V_\C \ar[r] & \Sp(H_\C) \ar[r] & 1
}
$$
commutes.

\begin{thm}
\label{thm:main2}
If $g \ge 6$, then there is an isomorphism $H_1(K_g,\C) \cong M$ which is
equivariant with respect to a suitable choice of the Zariski dense homomorphism
$\G_g/K_g \to \Sp(H_\C)\ltimes V_\C$ described above.
\end{thm}

\subsection{Relative completion}
\label{ss10}

These results are proved using the {\em infinitesimal Alexander invariant}
introduced in \cite{PS0} and the {\em relative completion} of mapping class
groups from \cite{H}. Alexander invariants occur as $K_g$ contains the
commutator subgroup $T_g'$ of $T_g$ and $K_g/T_g'$ is a finite vector space over
$\Z/2\Z$. So one would expect $H_1(K_g,\C)$ to be closely related to the
complexified Alexander invariant $H_1(T_g',\C)$ of $T_g$. A second step is to
replace $T_g$ by its Malcev (i.e., unipotent) completion, and $K_g$ by the
derived subgroup of the unipotent completion of $T_g$. These groups, in turn,
are replaced by their Lie algebras. The resulting module is the infinitesimal
Alexander invariant of $T_g$. 

The role of relative completion of mapping class groups is that it allows, via
Hodge theory, to identify filtered invariants, such as $H_1(T_g', \C)$, with
their associated graded modules which are, in general, more amenable to
computation. For example, the lower central series of $T_g$ induces via
conjugation a filtration on $H_1(T_g', \C)$, whose first graded piece is
identified in \cite{H} with $V(2\lambda_2) \oplus \C$, over $\Sp(H_\C)$.

\subsection{Alexander invariants}

The classical Alexander invariant of a group $G$ is the abelianization
$G_{\ab}'$ of its derived subgroup $G':=[G,G]$. Conjugation by $G$ endows it
with the structure of a module over the integral group ring $\Z G_{\ab}$
of the abelianization of $G$. More generally, if $N$ is a normal subgroup of $G$
that contains $G'$, then one has the $\C G_{\ab}$-module
$N_{\ab}\otimes \C = H_1(N,\C)$. Our primary example is where $G=T_g$ and $N$ is
its Johnson subgroup $K_g$.

There is an infinitesimal analog of the Alexander invariant. It is obtained by
replacing the group $G$ by its (complex) Malcev completion $\ccG(G)$ (also known
as its unipotent completion). The Malcev completion of $G$ is a prounipotent
group, and is thus determined by its Lie algebra $\cG(G)$ via the exponential
mapping $\exp : \cG(G) \to \ccG(G)$, which is a bijection. (Cf.\ \cite[Appendix
A]{Q}.) The first version of the infinitesimal Alexander invariant of $G$ is the
abelianization $\ccB(G) := \ccG(G)'_{\ab}$ of the derived subgroup
$\ccG(G)'=[\ccG(G),\ccG(G)]$ of $\ccG(G)$. One also has the abelianization
$\bb(G)$ of the derived subalgebra $\cG(G)'= [\cG(G),\cG(G)]$ of $\cG(G)$. The
exponential mapping induces an  isomorphism $\bb(G) \to \ccB(G)$. When $N/G'$ is
finite, one has the diagram
\begin{equation}
\label{eqn:comparisons}
\xymatrix{
N_{\ab}\otimes\C \ar[r] & \ccB(G) & \bb(G) \ar[l]_\simeq^\exp 
}
\end{equation}
where the left most map is induced
by the homomorphism $N \to \ccG(G)'$.

The next step is to replace the Alexander invariant of $\cG(G)$ by a graded
module by means of the lower central series. Recall the lower central
series of a group $G$,
$$
G=G^1 \supseteq G^2 \supseteq G^3 \supseteq \cdots
$$
It is defined by $G^{q+1} = [G,G^q]$. One also has the lower central series 
$\{\cG(G)^q\}_{q\ge 1}$ of its Malcev Lie algebra. There is a natural graded Lie
algebra isomorphism
$$
\cG_\dot(G) :=\bigoplus_{q\ge 1} (G^q/G^{q+1})\otimes \C
\stackrel{\simeq}{\longrightarrow}
\bigoplus_{q\ge 1} \cG(G)^q/\cG(G)^{q+1},
$$
where the bracket of the left-hand side is induced by the commutator of $G$.

The {\em infinitesimal Alexander invariant} $\bb_{\bullet}(G)$ of $G$, as
introduced in \cite{PS0}, is the Alexander invariant of this graded Lie algebra,
with a degree shift by $2$:\footnote{That is, as a graded vector space,
$\bb_{q}(G) = \cG_{q+2} (G)'_{\ab}$, for $q\ge 0$.}
$$
\bb_{\bullet}(G) := \cG_{\bullet} (G)'_{\ab} [2]\, .
$$
The adjoint action induces an action of the abelian Lie algebra $\cG_{\bullet}
(G)_{\ab}= G_{\ab}\otimes \C$ on $\bb_{\bullet}(G)$, and this makes the latter a
(graded) module over the polynomial ring $\Sym_{\bullet} (G_{\ab}\otimes \C)$.
One reason for considering $\bb_\dot(G)$ is that, in general, it is easier to
compute than $\bb(G)$.

The invariant $\bb_\dot(G)$ is most useful when $G$ is a group whose Malcev Lie
algebra $\cG(G)$ is isomorphic to the degree completion $\widehat{\cG_\dot(G)}$
of its associated graded Lie algebra. Groups that satisfy this condition include
the Torelli group $T_g$ when $g\ge 3$, which is proved in \cite{H}, and 1-formal
groups\footnote{In the sense of Dennis Sullivan \cite{S}. Note that $T_g$ is
1-formal when $g\ge 6$, but is not when $g=3$.} (such as K\"ahler groups). An
isomorphism of $\cG(G)$ with $\widehat{\cG_\dot(G)}$ induces an isomorphism of
the infinitesimal Alexander invariant $\ccB(G)$ with the degree completion of
$\bb_\dot(G)$.

When $G$ is finitely generated, $N/G'$ is finite and $H_1(N,\C)$ is a finite dimensional nilpotent $\C
G_{\ab}$-module, it follows from Proposition \ref{prop:bkiso} that all maps in (\ref{eqn:comparisons}) are
isomorphisms.

\subsection{Main general result} \label{ss11}

To emphasize the key features, it is useful to abstract the situation. Define
the {\em Johnson kernel} $K_G$ of a group $G$ to be the kernel of the natural
projection, $G\surj G_{\abf}$, where $G_{\abf}$ denotes the maximal torsion-free
abelian quotient of $G$. Assume from now on that $G$ is finitely generated. For
example, when $g\ge 3$, the Torelli group $G=T_g$ is finitely generated,
$G_{\abf}=V_\Z$ and $K_G = K_g$.

Under additional assumptions, we want to relate the $\C G_{\ab}$-module
$H_1(K_G, \C)$ to the graded $\Sym_\dot(G_{\ab}\otimes \C)$-module
$\bb_\dot(G)$. The first issue is that the rings $\C G_{\ab}$ and
$\Sym_\dot(G_{\ab}\otimes\C)$ are different. 
This is not serious as it is well-known that they become isomorphic,
after completion. Specifically, denote the augmentation ideal of $\C G_{\ab}$ by
$I_{G_{\ab}}$ and the $I_{G_{\ab}}$-adic completion of $\C G_{\ab}$ by
$\widehat{\C G_{\ab}}$. The exponential mapping induces a filtered ring
isomorphism,
\begin{equation} \label{eq=expiso}
\widehat{\exp} \colon
\widehat{\C G_{\ab}} \stackrel{\simeq}{\longrightarrow} \widehat{\Sym_\dot
(G_{\ab}\otimes \C)}\, ,
\end{equation}
with the degree completion of $\Sym_\dot(G_{\ab}\otimes\C)$.

Recall that a $\C G_{\ab}$-module is {\em nilpotent} if it is annihilated by
$I^q_{G_{\ab}}$, for some $q$, and {\em trivial} if it is annihilated by
$I_{G_{\ab}}$. When $H_1(K_G, \C)$ is nilpotent, it has a {\em
natural} structure of  $\Sym_\dot(G_{\ab}\otimes \C)$-module. Indeed, $H_1(K_G,
\C)= \widehat{H_1(K_G, \C)}$ by nilpotence, so we may restrict via 
\eqref{eq=expiso} the canonical $\widehat{\Sym_\dot(G_{\ab}\otimes \C)}$-module
structure of $H_1(K_G, \C)$ to $\Sym_\dot(G_{\ab}\otimes \C)$. We may now state
our main general result.

\begin{thm} \label{thm:gral}
Suppose that $G$ is a finitely generated group whose Malcev Lie algebra $\cG(C)$
is isomorphic to the degree completion of its associated graded Lie algebra
$\cG_\dot(G)$. If $H_1(G, \C_{\rho})$ vanishes for every non-trivial character
$\rho : G \to \C^\ast$ that factors through $G_{\abf}$, then $H_1(K_G, \C)$ is a
finite-dimensional nilpotent $\C G_{\ab}$-module and there is a
$\Sym_\dot(G_{\ab}\otimes \C)$-module isomorphism
$
H_1(K_G, \C) \cong \bb_\dot(G).
$
Moreover, $I^q_{G_{\ab}}$ annihilates $H_1(K_G, \C)$ if and only if
$\bb_{q} (G)=0$. 
\end{thm}

The vanishing of $H_1(G,\C_\rho)$ above can be expressed geometrically in terms
of the {\em character group} $\T(G)=\Hom (G_{\ab}, \C^*)$ of $G$. Since $G$ is
finitely generated, this is an algebraic torus. Its identity component $\T^0(G)$
is the subtorus $\Hom (G_{\abf}, \C^*)$. The {\em restricted characteristic
variety} $\V (G)$ is the set of those $\rho \in \T^0 (G)$ for
which $H_1(G, \C_{\rho}) \ne 0$. It is known that $\V (G)$ is a Zariski
closed subset of $\T^0 (G)$. (This follows for instance from work by
E.~Hironaka in \cite{Hi}.) The vanishing hypothesis in Theorem \ref{thm:gral}
simply means that $\V (G)$ is trivial, i.e., $\V (G) \subseteq \{ 1\}$.

Work by Dwyer and Fried \cite{DF} (as refined in \cite{PS}) implies that $\V
(G)$ is finite precisely when $H_1(K_G, \C)$ is  finite-dimensional. This
approach led in \cite{DP} to the conclusion that  $\dim_\C H_1(K_g,\C)< \infty$,
for $g\ge 4$. Further analysis (carried out in Section \ref{sec:vtor}) reveals
that  $\V (G)$ is trivial if and only if $H_1(K_G,\C)$ is nilpotent over $\C
G_{\ab}$.

We show in Section \ref{sec:compl} that the $I_{G_{\ab}}$-adic completions of
$H_1(G', \C)$ and $H_1(K_G, \C)$ are isomorphic. The triviality of $\V(G)$
implies that the finite-dimensional vector space $H_1(K_G, \C)$ is isomorphic to
its completion. On the other hand, the first hypothesis of Theorem
\ref{thm:gral} implies, via a result from \cite{DPS}, that the degree completion
of the infinitesimal Alexander invariant $\bb_\dot(G)$ is isomorphic to the
$I_{G_{\ab}}$-adic completion of $H_1(G',\C)$. The details appear in Section
\ref{sec:infalex}.

\bigskip


To prove Theorem \ref{thm:main1} we need to check that $\V (T_g) \subseteq \{
1\}$. This is achieved in two steps. Firstly, we improve one of the main results
from \cite{DP}, by showing that  $\V (T_g)$ is not just finite, but consists
only of torsion characters. This is done in a broader context, in Theorem
\ref{thm:bref}. In this theorem, the symplectic symmetry plays a key role: the
$\Sp(H_\Z)$-module $(T_g)_{\ab}$ gives a canonical action of $\Sp(H_\Z)$ on the
algebraic group $\T^0(T_g)$. We know from \cite{DP} that this action leaves the
restricted characteristic variety $\V (T_g)$ invariant. The second step is to
infer that actually  $\V (T_g) = \{ 1\}$. We prove this by using  a key result
due to Putman, who showed in \cite{P} that all finite index subgroups of $T_g$
that contain $K_g$ have the same first Betti number when $g\ge 3$.

A basic result from \cite{H}, valid for $g\ge 3$, guarantees that $T_g$
satisfies the assumption on the Malcev Lie algebra in Theorem~\ref{thm:gral}.
Theorem \ref{thm:main1} follows. Again by \cite{H}, the group  $T_g$ is {\em
$1$-formal}, when $g\ge 6$; equivalently, the graded Lie algebra $\cG_{\bullet}
(T_g)$ has a quadratic presentation.  Theorem \ref{thm:main2} follows from a
general result in \cite{PS0}, that associates to a quadratic presentation of the
Lie algebra $\cG_{\bullet}(G)$ a finite $\Sym_{\bullet} (G_{\ab}\otimes
\C)$-presentation for the infinitesimal Alexander invariant $\bb_{\bullet} (G)$.
When $g\ge 6$, we use the quadratic presentation of $\cG_{\bullet} (T_g)$
obtained in \cite{H}.

\section{Completion} \label{sec:compl}

We start in this section by establishing several general results, related to
$I$-adic completions of Alexander-type  invariants. We refer the reader to the
books by Eisenbud \cite[Chapter 7]{E} and Matsumura \cite[Chapter 9]{M}, for
background on completion techniques in commutative algebra. Throughout the
paper, we work with $\C$-coefficients, unless otherwise specified.  The {\em
augmentation ideal} of a group $G$, $I_G$, is the kernel of the $\C$-algebra
homomorphism, $\C G \to \C$, that sends each group element to $1$.

Let $N$ be a normal subgroup of $G$. Note that $G$-conjugation endows
$H_{\bullet}N$ with a natural structure of (left) module over the group algebra
$\C (G/N)$, and similarly for cohomology. If $N$ contains the derived subgroup
$G'$,  both $H_{\bullet} N$ and $H^{\bullet} N$ may be viewed as $\C
(G/G')$-modules, by restricting the scalars via the ring epimorphism $\C (G/G')
\surj \C (G/N)$. When $G$ is finitely generated, $H_1N$ is a finitely generated
module over the commutative Noetherian ring  $\C (G/N)$.

An important particular case arises when $N=G'$. Denoting abelianization by 
$G_{\ab}:=G/G'$, set $B(G):= H_1G'=G'_{\ab}\otimes \C = (G'/G'')\otimes \C$,
and call $B(G)$  the {\em Alexander invariant} of $G$.  These constructions are
functorial, in the following sense. Given a group homomorphism, $\varphi: \oG
\to G$,  it induces a $\C$-linear map, $B(\varphi): B(\oG)\to B(G)$, and a ring
homomorphism $\C \varphi: \C \oG_{\ab} \to \C G_{\ab}$. Moreover, $B(\varphi)$
is $\C \varphi$-{\em equivariant}, i.e.,  $B(\varphi) (\bar{a}\cdot \bar{x})= \C
\varphi (\bar{a})\cdot B(\varphi)(\bar{x})$, for $\bar{a}\in \C \oG_{\ab}$ and
$\bar{x}\in B(\oG)$.

The {\em $I$-adic filtration} of the $\C G_{\ab}$-module $B(G)$, $\{
I^q_{G_{\ab}} \cdot  B(G) \}_{q\ge 0}$, gives rise to the {\em completion} map
$B(G) \to \widehat{B(G)}$, and to the $I$-adic associated graded, $\gr_{\bullet}
B(G)$. By $\C \varphi$-equivariance, $B(\varphi)$ respects the $I$-adic
filtrations. Consequently, there is an induced filtered map, 
$\widehat{B(\varphi)}:  \widehat{B(\oG)} \to \widehat{B(G)}$, compatible with
the completion maps. One knows that  $\widehat{B(\varphi)}$ is a filtered
isomorphism if and only if $\gr_{\bullet}(B\varphi): \gr_{\bullet} B(\oG) \to
\gr_{\bullet} B(G)$ is an isomorphism.

A useful related construction (see \cite{Se}) involves the lower central series
of a group $G$. The (complex) {\em associated graded Lie algebra}
$$
\cG_\dot (G) := \bigoplus_{q\ge 1} (G^q /G^{q+1})\otimes \C
$$
is generated as a Lie  algebra by $\cG_1(G)= H_1 G$. Each group homomorphism
$\varphi: \oG \to G$ gives rise to a graded Lie algebra homomorphism,
$\gr_{\bullet}(\varphi): \cG_{\bullet}(\oG) \to \cG_{\bullet}(G)$.

Malcev completion (over $\C$), as defined by Quillen \cite[Appendix A]{Q}, is a
useful tool. It associates to a group $G$ a complex prounipotent group
$\ccG(G)$, and a homomorphism $G \to \ccG(G)$. The {\em Malcev Lie algebra} of
$G$ is the Lie algebra $\cG(G)$ of $\ccG(G)$. It is pronilpotent. The
exponential mapping $\exp : \cG(G) \to \ccG(G)$ is a bijection.

The lower central series filtrations
\begin{align*}
G &= G^1 \supseteq G^2 \supseteq G^3 \supseteq \cdots
\cr
\ccG(G) &= \ccG(G)^1 \supseteq \ccG(G)^2 \supseteq \ccG(G)^3 \supseteq \cdots
\cr
\cG(G) &= \cG(G)^1 \supseteq \cG(G)^2 \supseteq \cG(G)^3 \supseteq \cdots
\end{align*}
of $G$, $\ccG(G)$ and $\cG(G)$ are preserved by the canonical homomorphism $G
\to \ccG(G)$ and the exponential mapping $\exp : \cG(G) \to \ccG(G)$. They
induce Lie algebra isomorphisms of the associated graded objects:
$$
\xymatrix{
\gr_\dot (G)\otimes \C \ar[r]^\simeq & \gr_\dot \ccG(G) &
\gr_\dot \cG(G) \ar[l]_\simeq
}
$$
(cf. \cite[Appendix A]{Q}).

We will need the following basic fact, which is a straightforward generalization
of a result \cite{stallings} of Stallings: if a group homomorphism  $\varphi:
\oG \to G$ induces an isomorphism $\varphi^1: H^1G \isom H^1 \oG$ and a
monomorphism  $\varphi^2: H^2G \inj H^2 \oG$, then 
\begin{equation}
\label{eq:maliso}
\cG (\varphi): \cG (\oG) \isom \cG (G)
\end{equation}
is a filtered Lie isomorphism. A proof can be found in \cite[Corollary 3.2]{H1}.

With these preliminaries, we may now state and prove our first result.

\begin{prop}
\label{prop:bhatiso}
Suppose that $\oG$ is a finite index subgroup of a finitely generated group $G$.
If $\varphi_1: H_1 \oG \to H_1 G$ is a an isomorphism, then
$\widehat{B(\varphi)}:  \widehat{B(\oG)} \to \widehat{B(G)}$ is a filtered
isomorphism, where $\varphi: \oG \inj G$ is  the inclusion map.
\end{prop}

\begin{proof}
Since $[G:\oG]$ is finite, $\varphi^{\bullet}: H^{\bullet}G \to H^{\bullet}\oG$
is a monomorphism. So,  $\varphi^1$ is an isomorphism and $\varphi^2$
is injective. Hence, the filtered Lie isomorphism \eqref{eq:maliso} holds. 

Proposition 5.4 from \cite{DPS} guarantees that the filtered vector space
$\widehat{B(G)}$ is functorially determined by the filtered Lie algebra $\M
(G)$. This completes the proof.
\end{proof}

Consider now a group extension
\begin{equation}
\label{eq:ext}
1\to N \stackrel{\psi}{\longrightarrow}  \pi \to Q\to 1\, .
\end{equation}
Denote by $p_{\bullet}: H_{\bullet}N \surj (H_{\bullet}N)_Q$ the canonical
projection onto the co-invariants.  Clearly, $\psi_{\bullet}: H_{\bullet}N \to
H_{\bullet} \pi$ factors through $p_{\bullet}$, giving rise to a map
\begin{equation}
\label{eq:defq}
q_{\bullet}: (H_{\bullet}N)_Q \to H_{\bullet} \pi \, .
\end{equation}
When $Q$ is finite, $q_{\bullet}$ is an isomorphism; see Brown's book
\cite[Chapter III.10]{B}. 

Given a $\C \pi$-module $M$, note that $I_N \cdot M$ is a $\C \pi$-submodule of
$M$;  see \cite[Chapter II.2]{B}. Consequently, the natural projection onto the
$N$-co-invariants, $p: M\surj M_N$, is $\C \pi$-linear and induces a filtered
map,  $\hat{p}: \widehat{M} \to \widehat{M_N}$, between $I_{\pi}$-adic
completions. 

We will need the following probably known result. For  the reader's convenience,
we sketch a proof.

\begin{lemma}
\label{lem:phatiso}
Suppose that $N$ is a finite subgroup of a finitely generated abelian group
$\pi$. If $M$ is a finitely generated $\C \pi$-module, then $\hat{p}:
\widehat{M} \to \widehat{M_N}$ is a filtered isomorphism.
\end{lemma}

\begin{proof}
We start with a simple remark: if $R$ is a finitely generated commutative
$\C$-algebra and $I\subseteq R$ is a maximal ideal, then the roots of unity $u$
from $1+I$ act as the identity on $M/I^q \cdot M$, for all $q$, when $M$ is a
finitely generated $R$-module. Indeed, $u-1$ annihilates $I^s \cdot M/ I^{s+1}
\cdot M$ for all $s$, so the $u$-action on the finite-dimensional $\C$-vector
space $M/I^q \cdot M$ is both unipotent and semisimple, hence trivial.

Now, consider the exact sequence of finitely generated $R$-modules,
\[
0\to I_N \cdot M \rightarrow M \rightarrow M_N \to 0\, ,
\]
where $R=\C \pi$. Tensoring it with $R/I^q_{\pi}$, we infer that our claim is
equivalent to $I_N \cdot M \subseteq \cap_{q} I^q_{\pi} \cdot M$. This in turn
follows from the remark.
\end{proof}

\begin{remark}
\label{rem=restr}
Let $M$ be a module over a group ring $\C \pi$, and $\overline{\pi} \surj \pi$ a
group epimorphism, giving $M$ a structure of $\C \overline{\pi}$-module, by
restriction via $\C \overline{\pi} \surj \C \pi$.  Plainly,
$I^q_{\overline{\pi}} \cdot M =I^q_{\pi} \cdot M$, for all $q$. In particular,
the $I_{\overline{\pi}}$-adic and $I_{\pi} $-adic completions of $M$ are
filtered isomorphic, and $M$ is nilpotent (or trivial) over  $\C \overline{\pi}$
if and only if this happens over $\C \pi$.  
\end{remark}

Given a group $G$, set $G_{\abf}:=G_{\ab}/(\text{torsion})$. The {\em Johnson kernel}, $K_G$, is the kernel of the
canonical projection $G \surj G_{\abf}$. When $G=T_g$ and $g\ge 3$, Johnson's fundamental results from
\cite{J0, J3} show that $K_G=K_g$,  whence our terminology.

More generally, consider an extension
\begin{equation}
\label{eq:kext}
1\to G' \stackrel{\psi}{\longrightarrow} K \to F\to 1\, ,
\end{equation}
with $F$ finite. 
Plainly, $\psi_{\bullet}: H_{\bullet}G' \to H_{\bullet}K$ is $\C
G_{\ab}$-linear. Let $\widehat{\psi_{\bullet}}$ be the induced map on
$I_{G_{\ab}}$-adic completions. (When $K=K_G$, note that $ H_{\bullet}K_G$ is actually a $\C
G_{\abf}$-module, with $\C G_{\ab}$-module structure induced by restriction, via
$\C G_{\ab} \surj \C G_{\abf}$. By Remark \ref{rem=restr}, its
$I_{G_{\ab}}$-adic and $I_{G_{\abf}}$-adic completions coincide.) Here is our
second main result in this section. 

\begin{prop}
\label{prop:bkiso}
If $G$ is a finitely generated group and $K$ is a subgroup like in \eqref{eq:kext}, then 
$\widehat{\psi_{1}}: \widehat{H_1G'} \to \widehat{H_1 K}$ is a filtered isomorphism.
\end{prop}

\begin{proof}
We apply Lemma \ref{lem:phatiso} to $F\subseteq G_{\ab}$ and $M=H_1 G'$, to
obtain a filtered isomorphism  $\hat{p}: \widehat{H_1 G'}\isom \widehat{(H_1
G')_F}$ between $I_{G_{\ab}}$-adic completions. We conclude by noting that the
isomorphism \eqref{eq:defq} coming from \eqref{eq:kext}, $q: (H_1 G')_F \isom
H_1 K$, is $\C G_{\ab}$-linear. The last claim is easy to check: plainly, the
$\C F$-module structure on $H_1 G'$  coming from \eqref{eq:kext} is the
restriction to $\C F$ of the canonical $\C G_{\ab}$-structure.
\end{proof}

\section{Characteristic varieties} \label{sec:vtor}

We show that the (restricted) characteristic variety of $T_g$ is trivial for all
$g\ge 4$, as stated in Theorem \ref{thm:main1}, thus improving one of the main
results in \cite{DP}. Fix a symplectic basis of the first homology $H_\Z$ of
the reference surface $\Sigma$. This gives an identification of $\Sp(H_R)$
with $\Sp_g(R)$ for all rings $R$.

We start by reviewing a couple of definitions and relevant facts. Let $G$ be a
finitely generated group. The {\em character torus} $\T (G)= \Hom (G_{\ab},
\C^*)$ is a linear algebraic group with coordinate ring $\C G_{\ab}$. The
connected component of $1\in \T (G)$ is denoted $\T^0(G)= \Hom (G_{\abf}, \C^*)$
and has coordinate ring $\C G_{\abf}$.

The {\em characteristic varieties} of $G$ are defined for (degree) $i\ge 0$,
(depth) $k\ge 1$ by 
\begin{equation}
\label{eq:defv}
\V^i_k (G)= \{ \rho\in \T (G) \mid \dim_{\C} H_i(G, \C_{\rho})\ge k \}\, .
\end{equation}
Here $\C_{\rho}$ denotes the $\C G$-module $\C$  given by the change of rings
$\C G\to \C$ corresponding to $\rho$. Their {\em restricted} versions are the
intersections  $\V^i_k (G) \cap \T^0 (G)$. The restricted characteristic
variety  $\V^1_1 (G) \cap \T^0 (G)$ is denoted $\V (G)$. As explained in
\cite[Section 6]{DP}, it follows from results in \cite{Hi} about finitely
presented groups that both $\V^1_k (G)$ and $\V^1_k (G) \cap \T^0 (G)$ are
Zariski closed subsets, for all $k$.

When $G=T_g$ and $g\ge 3$, these constructions acquire an important symplectic
symmetry; see \cite{DP}. We recall that the linear algebraic group $\Sp_g (\C)$
is defined over $\Q$, simple, with positive $\Q$-rank, and contains $\Sp_g (\Z)$
as an arithmetic subgroup. 

The $\G_g$-conjugation in the defining extension \eqref{eq:deft} for $T_g$
induces representations of $\Sp_g (\Z)$ in the finitely generated abelian groups
$(T_g)_{\ab}$ and $(T_g)_{\abf}$. They give rise to natural $\Sp_g
(\Z)$-representations in the algebraic groups $\T (T_g)$ and  $\T^0 (T_g)$, for
which the inclusion $\T^0 (T_g) \subseteq \T (T_g)$ becomes $\Sp_g
(\Z)$-equivariant. Furthermore, $\V (T_g)\subseteq \T^0 (T_g)$ is $\Sp_g
(\Z)$-invariant. 

By Johnson's work \cite{J0, J3}, we also know that the $\Sp_g (\Z)$-action on
$(T_g)_{\abf}$  extends to a rational, irreducible and non-trivial $\Sp_g
(\C)$-representation in $(T_g)_{\abf} \otimes \C$. 

We will need the following refinement of a basic result on propagation of
irreducibility, proved by Dimca and Papadima in \cite{DP}. This refinement is 
closely related to an open question formulated in \cite[Section 10]{PS1}, on
outer automorphism groups of free groups.

\begin{theorem}
\label{thm:bref}
Let $L$ be a $D$-module which is finitely generated and free as an abelian
group. Assume that $D$ is an arithmetic subgroup of a simple $\C$-linear
algebraic group $S$ defined over $\Q$, with $\rank_\Q(S)\ge 1$. Suppose also
that the $D$-action on $L$ extends to an irreducible, non-trivial, rational
$S$-representation in $L\otimes \C$. Let $W\subset \T (L)$ be a $D$-invariant,
Zariski closed, proper subset of $\T (L)$. Then $W$ is a finite set of torsion
elements in $\T (L)$. \end{theorem}

\begin{proof}
According to one of the main results from \cite{DP} (which needs no
non-triviality assumption on the $S$-representation $L\otimes \C$), $W$ must be
finite. We have to show that any $t\in W$ is a torsion point  of $\T =\T (L)$.
We know that the stabilizer of $t$, $D_t$, has finite index in $D$. By Borel's
density theorem, $D_t$ is Zariski dense in $S$. 

Suppose that $t\in W$ has infinite order, and let $\T_t \subseteq \T$ be the
Zariski closure of the subgroup generated by $t$. By our assumption, the closed
subgroup $\T_t$ is positive-dimensional. Since $\T_t$ is fixed by $D_t$, the Lie
algebra $T_1 \T_t \subseteq T_1 \T= \Hom (L\otimes \C, \C)$ is $D_t$-fixed as
well. By Zariski density, $T_1 \T_t$ is then a non-zero, $S$-fixed subspace of
$T_1 \T$, contradicting the non-triviality hypothesis on $L\otimes \C$.
\end{proof}

We know from \cite{DP} that $\V (T_g)$ is finite, for $g\ge 4$. We may apply
Theorem \ref{thm:bref}  to $D= \Sp_g (\Z)$ acting on $L= (T_g)_{\abf}$, $S=
\Sp_g (\C)$ and $W= \V (T_g)$. We infer that $\V (T_g)$ consists of $m$-torsion
elements in $\T^0 (T_g)$, for some $m\ge 1$. 

To derive the triviality of $\V(T_g)$ from this fact, we will use another
standard tool from commutative algebra. For an affine $\C$-algebra $A$, let
$\Spec (A)$ be its maximal spectrum. For a $\C$-algebra map between affine
algebras, $f:A\to B$, $f^*: \spec (B)\to \spec (A)$  stands for the induced map,
that sends $\Spec (B)$ into $\Spec (A)$. For a finitely generated $A$-module
$M$, the {\em support} $\supp_A (M)$ is  the Zariski closed subset of $\spec
(A)$ $V(\ann (M))= V(E_0(M))$, where $E_0(M)$ is the ideal generated by the
codimension zero minors of a finite $A$-presentation for $M$; see \cite[Chapter
20]{E}. 

There is a close relationship between characteristic varieties and supports of
Alexander-type invariants. Let $N$ be a normal subgroup of a finitely generated
group $G$, with abelian quotient. Denote by $\nu: G\surj G/N$ the canonical
projection, and let $\nu^*: \T (G/N) \inj \T (G)$ be the induced map on maximal
spectra of corresponding abelian group algebras. It follows for instance from
Theorem 3.6 in \cite{PS} that $\nu^*$ restricts to an identification (away from
$1$)
\begin{equation}
\label{eq:supv}
\Spec (\C (G/N)) \cap \supp_{\C (G/N)} (H_1 N) \equiv \im (\nu^*)
\cap \V^1_1 (G)\, .
\end{equation}

We will need the following (presumably well-known) result on supports. For the
sake of completeness, we include a proof.

\begin{lemma}
\label{lem:supp}
If $f: A\inj B$ is an integral extension of affine $\C$-algebras and  $M$ is a
finitely generated $B$-module, then $M$ is finitely generated over $A$ and
$\supp_A (M)= f^*(\supp_B (M))$,
$\Spec (A) \cap \supp_A (M)= f^*( \Spec (B) \cap \supp_B (M))$. 
\end{lemma}

\begin{proof}
The reader may consult \cite[Chapter 4]{E}, for background on integral
extensions. Clearly, $\ann_A (M)= A\cap \ann_B (M)$, and the extension $\bar{f}:
A/\ann_A (M) \inj B/\ann_B (M)$ is again integral. The inclusion $f^*(\supp_B
(M))\subseteq \supp_A (M)$  follows from the definitions. For the other
inclusion, pick any prime ideal $\fp$ containing $\ann_A (M)$. Since $\bar{f}$ 
induces a surjection on prime spectra, there is a prime ideal $\fq$ containing
$\ann_B (M)$, such that $f^*(\fq)=\fp$, and $\fq$ is maximal, if $\fp$ is
maximal. 
\end{proof}

The interpretation \eqref{eq:supv} for the closed points of the support of
Alexander-type invariants leads to the following key nilpotence test.

\begin{lemma}
\label{lem:niltest}
Let $\nu : G\surj H$ be a group epimorphism with finitely generated source,
abelian image and kernel $N$. Then the following are equivalent.
\begin{enumerate}

\item \label{nil1}
The $\C H$-module $H_1 N$ is nilpotent.

\item \label{nil2}
The inclusion $\Spec (\C H) \cap \supp_{\C H} (H_1 N) \subseteq \{1\}$ holds.

\item \label{nil3}
The intersection $\nu^* (\T (H)) \cap \V^1_1 (G)$ is contained in $\{ 1\}$.

\end{enumerate}
\end{lemma}

\begin{proof}
Note that $1\in \T (H)$ corresponds to the maximal ideal $I_H \subseteq \C H$.
With this remark, the equivalence \eqref{nil1} $\eqv$ \eqref{nil2} becomes an
easy consequence of the Hilbert Nullstellensatz. The equivalence \eqref{nil2}
$\eqv$ \eqref{nil3} follows directly from \eqref{eq:supv}.
\end{proof}

For a group $G$, we denote by $p_G : G \surj G_{\abf}$ the canonical projection.
Assume $G$ is finitely generated and fix an integer $m\ge 1$.  Denoting by
$\iota_m : G_{\abf} \inj G_{\abf}$ the multiplication by $m$, note that its
extension to group algebras, $\C \iota_m : \C G_{\abf} \inj \C  G_{\abf}$, is
integral, and the associated map on maximal spectra, $\C \iota_m^*: \T
(G_{\abf}) \to \T (G_{\abf})$, is the $m$-power map of the character group
$\T (G_{\abf})$. 

Let $p_m : G(m) \surj G_{\abf}$ be the pull-back of $p_G$ via $\iota_m$.
Clearly, $G(m)$ is a normal, finite index subgroup of $G$ containing the Johnson
kernel $K_G$, with inclusion denoted  $\varphi_m : G(m) \inj G$, and
\begin{equation} 
\label{eq:pull} 
p_G \circ \varphi_m =\iota_m \circ p_m \, .
\end{equation}

\begin{lemma}
\label{lem:samek}
Assume that all finite index subgroups of $G$ containing $K_G$ have the same
first Betti number. Then the following hold.

\begin{enumerate}

\item \label{331}
The map induced by $\varphi_m$ on $I$-adic completions,
$\widehat{B(\varphi_m)}:  \widehat{B(G(m))} \isom \widehat{B(G)}$, is 
a filtered isomorphism.

\item \label{332}
The inclusion $K_{G(m)} \subseteq K_G$ is actually an equality. 

\end{enumerate}
\end{lemma}

\begin{proof}
Our assumption implies that $\varphi_m$ induces an isomorphism $H_1 G(m) \isom
H_1 G$. Property \eqref{331} follows then from  Proposition \ref{prop:bhatiso}. 
The second claim is a consequence of the fact that $p_m$ may be identified with
$p_{G(m)}$. To obtain this identification, we apply to \eqref{eq:pull} the
functor $\abf$. By construction, $\abf (p_G)$ is an isomorphism and $\abf
(\iota_m)=\iota_m$ is a rational isomorphism. We also know that $\abf
(\varphi_m) \otimes \Q : H_1 (G(m), \Q) \isom H_1 (G, \Q)$ is an isomorphism. We
infer that $\abf (p_m)$ is  a rational isomorphism, hence an isomorphism.
\end{proof}

{\bf Proof of Theorem \ref{thm:main1} (except for the non-triviality
assertion).} We first prove that $\V(T_g) = \{1\}$.  According to a recent
result of Putman \cite{P}, the group $G=T_g$ satisfies for $g\ge 3$  the
hypothesis of Lemma \ref{lem:samek}. Denote by $\psi: G' \inj K_G$ and $\psi_m :
G(m)' \inj K_{G(m)} =K_G$ the inclusions. According to Proposition
\ref{prop:bkiso},  they induce filtered isomorphisms,  $\widehat{B(G)} \isom
\widehat{H_1 K_G}$ and $\widehat{B(G(m))} \isom \widehat{H_1 K_G}$, between the
corresponding $I$-adic completions. In these isomorphisms, the $\C
G_{\abf}$-module  structure of  $M=H_1 K_G$ comes from the group extension
associated to $p_G$, respectively $p_m$. Denote the second  $\C G_{\abf}$-module
by ${}^m M$, and note that ${}^m M$ is obtained from $M$ by restriction of
scalars, via $\C \iota_m : \C G_{\abf} \inj \C  G_{\abf}$. 

Taking into account the isomorphism from Lemma \ref{lem:samek}\eqref{331}, it
follows that  $\id_M: {}^m M \to M$, viewed as a $\C \iota_m$-equivariant map,
induces an isomorphism between the corresponding $I$-adic completions. 

Let $\rho \in \T (G_{\abf})$ be a closed point of $\supp_{\C G_{\abf}} (M)$. By
\eqref{eq:supv}, applied to $N=K_G$, $\C \iota_m^* (\rho)=1$, since $\V (G)$
consists of  $m$-torsion points, for $g\ge 4$. We infer from Lemma
\ref{lem:supp}, applied to $f= \C \iota_m : \C G_{\abf} \inj \C  G_{\abf}$,
that   $\Spec (\C G_{\abf}) \cap \supp_{\C G_{\abf}} ({}^m M) \subseteq \{
1\}$.  Take $\nu =p_m : G(m) \surj G_{\abf}$ in Lemma \ref{lem:niltest}, whose
kernel is $K_{G(m)} =K_G$. We deduce that ${}^m M$ is nilpotent over $\C
G_{\abf}$, that is, $I^q \cdot {}^m M =0$ for some $q$, where $I\subseteq \C
G_{\abf}$ is the augmentation ideal. 

Denote by $\kk_m : M\to \widehat{{}^m M}$ and $\kk: M\to \widehat{M}$ the
completion maps, with kernels  $\cap_{r\ge 0} I^{r} \cdot {}^m M$ and 
$\cap_{r\ge 0} I^{r} \cdot M$. It follows from naturality of completion that 
$\kk$ is injective, since $\kk_m$ is injective and $\widehat{\id_M}:
\widehat{{}^m M} \isom \widehat{M}$ is an isomorphism. 

We also know from \cite{DP} that $\dim_{\C} M<\infty$. It follows that the
$I$-adic filtration of $M$  stabilizes to $I^{q} \cdot M = \cap_{r\ge 0} I^{r}
\cdot M= 0$, for $q$ big enough.   Applying Lemma \ref{lem:niltest} to $\nu =p_G
: G \surj G_{\abf}$, with kernel $K_G$, we infer that $\V (G)= \{ 1\}$.

\medskip 

We extract from the preceding argument the following corollary. Together with
the triviality of $\V(T_g)$, this completes the proof of Theorem~\ref{thm:main1}
(except for the non-triviality assertion) via Remark \ref{rem=restr}.

\begin{corollary}
\label{cor:halfm2}
If $g\ge 4$, then $H_1(K_g, \C)$ is a nilpotent module,
over both $\C (T_g)_{\ab}$ and $\C (T_g)_{\abf}$. 
\end{corollary}

\section{Infinitesimal Alexander invariant} \label{sec:infalex}

Our next task is to prove Theorem \ref{thm:main2} and the non-triviality
assertion of Theorem~\ref{thm:main1}. These follow from general results about
infinitesimal Alexander invariants. 

Let $\cH_{\bullet}$ be a positively graded Lie algebra. Consider
the exact sequence of graded Lie algebras
\begin{equation}
\label{eq:definf}
0\to \cH'_{\bullet}/\cH''_{\bullet} \rightarrow \cH_{\bullet}/\cH''_{\bullet} 
\rightarrow \cH_{\bullet}/\cH'_{\bullet} \to 0 \, .
\end{equation}
The universal enveloping algebra of the abelian Lie algebra
$\cH_{\bullet}/\cH'_{\bullet}$ is the graded polynomial algebra  $\Sym_{\bullet}
(\cH_{\ab})$. (When the Lie algebra $\cH_{\bullet}$  is generated by $\cH_{1}$,
$\Sym_{\bullet} (\cH_{\ab})= \Sym_{\bullet} (\cH_{1})$,  with the usual
grading.) The adjoint action in \eqref{eq:definf} yields a natural graded 
$\Sym_{\bullet} (\cH_{\ab})$-module structure on $\cH'_{\ab}$. It will be
convenient to shift degrees  and define the {\em infinitesimal Alexander
invariant} $\bb_{\bullet} (\cH) := \cH'_{\ab}[2]$, by analogy with the Alexander
invariant of a group. The graded vector space  $\bb_{\bullet} (\cH) =
\oplus_{q\ge 0} \bb_{q} (\cH)$, where $\bb_{q} (\cH)= \cH'_{q+2}/\cH''_{q+2}$,
becomes in this way a graded module over $\Sym_{\bullet} (\cH_{\ab})$. 

When $\cH_{\bullet}= \cG_{\bullet}(G)$, we denote the graded $\Sym_{\bullet}
(G_{\ab} \otimes \C)$-module $\bb_{\bullet} (\cH)$ by  $\bb_{\bullet} (G)$. Note
that the {\em degree filtration} of $\bb_{\bullet} (G)$, $\{ \bb_{\ge q} (G)
\}_{q\ge 0}$, coincides with its $(G_{\ab} \otimes \C)$-adic filtration, where 
$(G_{\ab} \otimes \C)$ is the ideal of $\Sym (G_{\ab} \otimes \C)$ generated by
$G_{\ab} \otimes \C$.

The infinitesimal Alexander invariant, introduced and studied in \cite{PS0}, is
functorial in the following sense. A graded Lie map $f: \cH \to \cK$ obviously
induces a degree zero map  $\bb_{\bullet} (f): \bb_{\bullet}(\cH) \to
\bb_{\bullet}(\cK)$, equivariant with respect to the graded algebra map $\Sym
(f_{\ab}): \Sym (\cH_{\ab}) \to \Sym (\cK_{\ab})$.  

Let $\bL_{\bullet}(V)$ be the free graded Lie algebra on a finite-dimensional
vector space $V$, graded by bracket length. Use the Lie bracket to identify
$\bL_{2}(V)$ and $\wedge^2 V$. For a subspace $R\subseteq \wedge^2 V$, consider
the (quadratic) graded Lie algebra $\cG= \bL (V)/\; \text{ideal}\; (R)$, with
grading inherited from $\bL_{\bullet}(V)$. 
Denote by  $\iota :R\inj \wedge^2 V$  the inclusion.

Theorem 6.2 from \cite{PS0} provides the following finite, free
$\Sym_{\bullet}(V)$-presentation  for the infinitesimal Alexander invariant: 
$\bb_{\bullet} (\cG) =\coker ({\nabla})$, where
\begin{equation}
\label{eq:presinf}
\nabla:= \id \otimes \iota + \delta_3 \colon \Sym_{\bullet}(V) \otimes
(R\oplus \wedge^3 V) \rightarrow \Sym_{\bullet}(V) \otimes \wedge^2 V\, ,
\end{equation}
$R$, $\wedge^3 V$ and $\wedge^2 V$ have degree zero, and the $\Sym (V)$-linear
map $\delta_3$ is given by $\delta_3 (a\wedge b\wedge c)= a\otimes b\wedge c +
b\otimes c\wedge a + c\otimes a\wedge b$, for $a,b,c\in V$. 

We begin by simplifying the presentation \eqref{eq:presinf}. To this end, let
$\beta: \wedge^2 \cG_1 \surj \cG_2$ be the Lie bracket.

\begin{lemma}
\label{lem:simpl}
For any quadratic graded Lie algebra $\cG$, $\bb_{\bullet} (\cG)= \coker
(\overline{\nabla})$,  as graded $\Sym (V)$-modules,  where the $\Sym
(V)$-linear map
\[
\overline{\nabla}: \Sym (V) \otimes  \wedge^3 \cG_1 \rightarrow \Sym (V)
\otimes \cG_2
\]
is defined by $\overline{\nabla}= (\id \otimes \beta) \circ \delta_3$. 
\end{lemma}

\begin{proof}
It is straightforward to check that the degree zero $\Sym (V)$-linear map $\id
\otimes\beta$  induces an isomorphism $\coker (\nabla) \isom \coker
(\overline{\nabla})$.
\end{proof}

{\bf Proof of Theorem \ref{thm:gral}.} In Lemma \ref{lem:niltest}, let $\nu$ be
the canonical projection $G \surj G_{\abf}$, with kernel $K_G$. By our
hypothesis on $\V (G)$ and Remark \ref{rem=restr},  we infer that the module
$H_1 K_G$ is nilpotent, over both $\C G_{\abf}$ and $\C G_{\ab}$. Therefore,
$\dim_{\C} H_1 K_G<\infty$ (since $H_1K_G$ is finitely generated over $\C
G_{\abf}$) and the $I_{G_{\ab}}$-adic completion map
\begin{equation}
\label{eq:griso1}
H_1 K_G \isom \widehat{H_1 K_G}
\end{equation}
is a filtered isomorphism. Proposition \ref{prop:bkiso} provides another
filtered isomorphism,
\begin{equation}
\label{eq:griso2}
\widehat{B(G)}\isom \widehat{H_1 K_G}\, ,
\end{equation}
between $I_{G_{\ab}}$-adic completions. 
A third filtered isomorphism is a consequence of our assumption on $\M (G)$:
\begin{equation}
\label{eq:griso3}
\widehat{B(G)}\isom \widehat{\bb_{\bullet} (G)}\, ,
\end{equation}
where the completion of $\bb_{\bullet} (G)$ is taken with respect to the degree
filtration; see \cite[Proposition 5.4]{DPS}.  Since $\dim_{\C}
\widehat{\bb_{\bullet} (G)}<\infty$,  we deduce that  $\dim_{\C} \bb_{\bullet}
(G)<\infty$.  Hence, the degree filtration is finite, and the completion map
\begin{equation}
\label{eq:griso4}
\bb_{\bullet} (G) \isom  \widehat{\bb_{\bullet} (G)}
\end{equation}
is a filtered isomorphism. 

By construction, the isomorphism \eqref{eq:griso1} is equivariant with respect
to the $I_{G_{\ab}}$-adic completion homomorphism, $\C G_{\ab} \to \widehat{\C
G_{\ab}}$. Again by construction, the isomorphism \eqref{eq:griso2} is
$\widehat{\C G_{\ab}}$-linear. By Proposition 5.4 from \cite{DPS}, the
isomorphism \eqref{eq:griso3} is $\widehat{\exp}$-equivariant, where
$\widehat{\exp} \colon \widehat{\C G_{\ab}} \isom \widehat{\Sym (G_{\ab}\otimes
\C)}$ is the identification \eqref{eq=expiso}. Since the degree filtration of
$\bb_{\bullet} (G)$ coincides with its $(G_{\ab} \otimes \C)$-adic filtration,
as noted earlier,  the isomorphism \eqref{eq:griso4} is plainly equivariant 
with respect to the $(G_{\ab} \otimes \C)$-adic completion homomorphism,  $\Sym
(G_{\ab}\otimes \C) \to \widehat{\Sym (G_{\ab}\otimes \C)}$. Putting these facts
together, we deduce from \eqref{eq:griso1}-\eqref{eq:griso4} that the natural
$\Sym (G_{\ab}\otimes \C)$-module structure of the nilpotent $\C G_{\ab}$-module
$H_1 K_G$, explained in the Introduction, is isomorphic to $\bb_{\bullet} (G)$
over $\Sym (G_{\ab}\otimes \C)$, as stated in Theorem \ref{thm:gral}.

To finish the proof of Theorem \ref{thm:gral}, we have to show that
$I_{G_{\ab}}^q\cdot H_1 K_G =0$ if and only if $\bb_{q}(G)=0$, for any $q\ge 0$.
This assertion will follow from the easily checked remark that, given a vector
space $M$ endowed with a decreasing Hausdorff filtration  $\{ F_r \}_{r\ge 0}$ (
i.e., $\cap_r F_r =0$), $F_q=0$ if and only if $\gr_{\ge q}(M)=\oplus_{r\ge q}
F_r/F_{r+1} =0$. Plainly, all maps \eqref{eq:griso1}-\eqref{eq:griso4} induce
isomorphisms at the associated graded level, and all filtrations are Hausdorff.
We deduce that $I_{G_{\ab}}^q\cdot H_1 K_G =0$ if and only if $\bb_r (G)=0$ for
$r\ge q$. Since $\bb_{\bullet} (G)$ is generated in degree zero over $\Sym
(G_{\ab}\otimes \C)$, this is equivalent to $\bb_q (G)=0$. The proof of Theorem
\ref{thm:gral} is complete.  \hfill $\square$

\medskip

{\bf Proof of the non-triviality assertion of Theorem~\ref{thm:main1}.} The
group $G=T_g$ satisfies the hypotheses of Theorem \ref{thm:gral} when $g\ge 4$.
Consequently, if $H_1 K_g$ is a trivial $\C (T_g)_{\ab}$-module, then $\bb_1
(T_g)= \cG_3 (T_g)=0$.  This implies that $\cG_{\ge 3} (T_g)=0$, since the Lie
algebra $\cG_{\bullet} (T_g)$ is generated in degree $1$. In particular,
$\dim_{\C} \cG_{\bullet} (T_g)< \infty$, which contradicts Proposition 9.5 from
\cite{H}. \hfill $\square$

\medskip

For the proof of Theorem \ref{thm:main2}, we need to recall the
main result of Hain from \cite{H}, that gives an explicit presentation of the
graded Lie algebra $\cG_{\bullet} (T_g)$, for $g\ge 6$, in 
representation-theoretic terms. For representation theory, we follow the
conventions from Fulton and Harris \cite{FH},  like in \cite{H}.

The conjugation action in \eqref{eq:deft} induces an action of $\Sp_g (\Z)$ on
$\cG_{\bullet} (T_g)$, by graded Lie algebra automorphisms. By Johnson's work,
the $\Sp_g (\Z)$-action on $\cG_{1} (T_g)$ extends to an irreducible rational
representation of $\Sp_g (\C)$. It follows that the $\Sp_g (\Z)$-action on
$\cG_{\bullet} (T_g)$ extends to a degree-wise rational representation of $\Sp_g
(\C)$, by graded Lie algebra automorphisms. By naturality, the symplectic Lie
algebra $\sp_g (\C)$ acts on $\bb_{\bullet} (T_g)$.

The fundamental weights of $\sp_g (\C)$ are denoted $\lambda_1, \dots,
\lambda_g$.  The irreducible finite-dimensional representation with highest
weight $\lambda= \sum_{i=1}^g n_i \lambda_i$ is denoted $V(\lambda)$. By
Johnson's work, $\cG_1 (T_g)= V(\lambda_3):= V$. The irreducible decomposition
of the $\sp_g (\C)$-module  $\wedge^2 V(\lambda_3)$ is of the form $\wedge^2
V(\lambda_3) = R\oplus V(2\lambda_2)\oplus V(0)$, with all multiplicities equal
to $1$. For $g\ge 6$, $\cG_{\bullet} := \cG_{\bullet} (T_g)= \bL_{\bullet}
(V)/\; \text{ideal}\; (R)$,  as graded Lie algebras with $\sp_g (\C)$-action. In
particular, $\beta: \wedge^2 \cG_1 \surj \cG_2$ is  identified with the
canonical $\sp_g (\C)$-equivariant projection  $\wedge^2 V(\lambda_3)\surj 
V(2\lambda_2)\oplus V(0)$.

Set $V(0)= \C \cdot z$, $\tilde{R}= R+ \C \cdot z$, and denote by $\pi :
\wedge^2 V(\lambda_3)\surj  V(2\lambda_2)$ the canonical $\sp_g
(\C)$-equivariant projection. Note that both  $\id \otimes \pi : \Sym
(V(\lambda_3)) \otimes \wedge^2 V(\lambda_3) \to  \Sym (V(\lambda_3)) \otimes
V(2\lambda_2)$  and the map $\delta_3:  \Sym (V(\lambda_3)) \otimes \wedge^3
V(\lambda_3) \to \Sym (V(\lambda_3)) \otimes \wedge^2 V(\lambda_3)$  from
\eqref{eq:presinf} are $\sp_g (\C)$-linear. Consequently, 
\begin{equation}
\label{eq:tildepres}
\widetilde{\nabla}:= (\id \otimes \pi) \circ \delta_3 \colon
\Sym (V(\lambda_3)) \otimes \wedge^3 V(\lambda_3) \rightarrow
\Sym (V(\lambda_3)) \otimes V(2\lambda_2)
\end{equation}
is both $\Sym (V(\lambda_3))$-linear and $\sp_g (\C)$-equivariant. We are going
to view the $\sp_g (\C)$-trivial module $\C \cdot z$ as a trivial
$\Sym_{\bullet} (V(\lambda_3))$-module concentrated in degree zero, and assign
degree $0$ to both $\wedge^3 V(\lambda_3)$ and $V(2\lambda_2)$. 

Consider the canonical, $\sp_g (\C)$-equivariant graded Lie epimorphism
\begin{equation}
\label{eq:deff}
f: \cG_{\bullet}=  \bL_{\bullet} (V)/\; \text{ideal}\; (R) \surj 
\bL_{\bullet} (V)/\; \text{ideal}\; (\tilde{R})= \cK_{\bullet}\, .
\end{equation}

\begin{lemma}
\label{lem:tildesimpl}
The induced $\Sym (V)$-linear,  $\sp_g (\C)$-equivariant map $\bb_{\bullet} (f)$
is onto,  with $1$-dimensional kernel $\C \cdot z$.
\end{lemma}

\begin{proof}
The first three claims are obvious. It is equally clear that $\bb_0 (f)$ has
kernel $\C \cdot z$.  To prove injectivity in degree $q\ge 1$, start with the
class $\bar{x}$ of an arbitrary element $x\in \bL_{q+2} (V)$. If $\bb
(f)(\bar{x})=0$, then $x$ is equal, modulo $\bL'' (V)$, with a linear
combination of Lie monomials of the form $\ad_{v_1}\cdots
\ad_{v_{q}}(\tilde{r})$, with $\tilde{r}\in \tilde{R}$. 

Therefore, $\bar{x}$ belongs to the $\C$-span of elements of the form
$\overline{\ad_{v_1}\cdots \ad_{v_{q}}(z)}$. As shown in \cite{H}, the class of
$z$ is a central element of the Lie algebra $\bL_{\bullet} (V)/\; \text{ideal}\;
(R)$,  and so we are done, since $q\ge 1$.
\end{proof}

{\bf Proof of Theorem \ref{thm:main2}.} By Theorem \ref{thm:gral}, $H_1K_g= \bb
(\cG)$, over $\Sym (V)$, with $\cG$ as in \eqref{eq:deff}. By Lemma
\ref{lem:tildesimpl},  $\bb (\cG)= \bb (\cK) \oplus \C \cdot z$, as graded $\Sym
(V)$-modules, where $\C \cdot z$ is $\Sym (V)$-trivial (since $z$ is central in
$\cG$), with degree $0$. 

By Lemma \ref{lem:simpl}, the graded $\Sym (V)$-module $\bb_{\bullet} (\cG)$ has
presentation \eqref{eq:presintro}; see \eqref{eq:tildepres}. Note also that the
identification $\bb (\cG)= \bb (\cK) \oplus V(0)$ is compatible with the natural
$\sp_g (\C)$-symmetry of $\bb (\cG)= \bb (T_g)$.

It remains to prove the assertion about the action of $\G_g/K_g$ on
$H_1(K_g,\C)$. For this we use the theory of relative completion of mapping
class groups developed and studied in \cite{H}. Denote the completion of the
mapping class group with respect to the standard homomorphism $\G_g \to
\Sp_g(\C)$ by $\R(\G_g)$. Right exactness of relative completion implies that
there is an exact sequence
$$
\ccG(T_g) \to \R(\G_g) \to \Sp_g(\C) \to 1
$$
such that the diagram
$$
\xymatrix{
1 \ar[r] & T_g \ar[r]\ar[d] & \G_g \ar[r]\ar[d] & \Sp_g(\Z) \ar[r]\ar[d] & 1 \cr
 & \ccG(T_g) \ar[r] & \R(\G_g) \ar[r] & \Sp_g(\C) \ar[r] & 1
}
$$
commutes, where $\ccG(T_g)$ denotes the Malcev completion of $T_g$.

The conjugation action of $\G_g$ on $T_g$ induces an action of $\G_g$ on the
Malcev Lie algebra $\cG(T_g)$ of the Torelli group. Basic properties of relative
completion imply that this action factors through the natural homomorphism $\G_g
\to \R(\G_g)$. This action descends to an action of $\R(\G_g)$ on the Alexander
invariant $\bb(T_g)$ of $\cG(T_g)$. Its kernel contains the image of
$\ccG(T_g)'$ in $\R(\G_g)$. Basic facts about the Lie algebra of $\R(\G_g)$
given in \cite{H} imply that, when $g\ge 3$, $\R(\G_g)/\im \ccG(T_g)'$ is an
extension
$$
1 \to V \to \R(\G_g)/\im \ccG(T_g)' \to \Sp_g(\C) \to 1.
$$
Levi's theorem implies that this sequence is split. However, we have to choose
compatible splittings of the lower central series of $\cG(T_g)$ and this
sequence. The existence of such compatible splittings is a consequence of the
existence of the mixed Hodge structures on $\cG(T_g)$ and on the Lie algebra of
$\R(\G_g)$, and the fact that the weight filtration of $\cG(T_g)$ (suitably
renumbered) is its lower central series. Such compatible mixed Hodge structures
are determined by the choice of a complex structure on the reference surface
$\Sigma$. With such compatible splittings, one obtains a commutative diagram
$$
\xymatrix{
\G_g/K_g \ar[r] &
\R(\G_g)/\im \ccG(T_g)'\ar[d] & \Sp_g(\C)\ltimes V \ar[l]_(0.45)\simeq \ar[d]\cr
& \Aut(\bb(T_g)) & \Aut(\bb_\dot(T_g)) \ar[l]_\simeq
}
$$
when $g\ge 4$.
This completes the proof of Theorem \ref{thm:main2}. \hfill $\square$

\medskip

\begin{example}
\label{ex:kfree}
Let us examine the simple case when $G=F_n$, the non-abelian free group on $n$
generators. In this case, $H_1(K_G, \C)= B(G)\otimes \C$. Since $G$ is
$1$-formal, Theorem 5.6 from \cite{DPS} identifies the $I$-adic completion
$\widehat{H_1 K_G}$ with the degree completion  $\widehat{\bb_{\bullet} (\cG)}$,
where $\cG =\cG_{\bullet}(G)= \bL_{\bullet} (V)$, and $V=H_1 (F_n, \C)= \C^n$. 

On the other hand, $\V^1_1 (G) =\V^1_1 (G)\cap \T^0(G)= (\C^*)^n$ is infinite,
in contrast with the setup from Theorem \ref{thm:gral}. It follows from
Corollary 6.2 in \cite{PS} that  $\dim_{\C} H_1 K_G =\infty$. It is also
well-known  that $\dim_{\C} \bb_{\bullet} (\cG) =\infty$ when $n>1$.

This non-finiteness property of $\bb_{\bullet} (\cG)$ can be seen
concretely by using the exact Koszul complex, $\{ \delta_i : P_{\bullet}\otimes
\wedge^i V \to  P_{\bullet}\otimes \wedge^{i-1} V \}$, where $P_{\bullet}= \Sym
(V)$. Indeed, we infer from \eqref{eq:presinf} that, for every $q\ge 0$,
$$
\bb_q (\cG) = \coker (\delta_3 : P_{q-1}\otimes \wedge^3 V \to 
P_{q}\otimes \wedge^{2} V) \cong
\ker (\delta_1 : P_{q+1}\otimes V \surj  P_{q+2})
$$ 
has dimension $\binom{q+n}{q+2} (q+1)$, a computation that goes back to Chen's
thesis \cite{C}. Note also that each $\bb_q (\cG)$ is an $\sl_n (\C)$-module. It
turns out that these modules   are irreducible, as we now explain. 

Let $\{ \lambda_1, \dots, \lambda_{n-1}\}$ be the set of fundamental weights of
$\sl_n (\C)$ associated to the ordered basis $e_1,\dots,e_n$ of $V$, as in
\cite{FH}. One can easily check that, for each $q\ge 0$, the image $v$ of the
vector
$$
u= e_1^q \otimes ( e_1\wedge e_2) \in P_q\otimes \wedge^2 V
$$
in $\bb_q (\cG)$ is non-zero. Since $u$ is a highest weight vector of weight
$q\lambda_1 + \lambda_2$, it follows that $v$ generates a copy of the
irreducible $\sl_n (\C)$-module $V(q\lambda_1 + \lambda_2)$ in $\bb_q (\cG)$.
Since $\dim V(q\lambda_1 + \lambda_2) = \dim \bb_q (\cG)$, we conclude that
$$
\bb_q (\cG)= V(q\lambda_1 + \lambda_2).
$$
\end{example}

\begin{ack}
We thank the referee for suggestions and comments that resulted in an
improved exposition. The third author is grateful to the Max-Planck-Institut f\" ur Mathematik (Bonn),
where he completed this work, for hospitality and the excellent research atmosphere.
\end{ack}

\newcommand{\arxiv}[1]{{arXiv:#1}}

\bibliographystyle{amsplain}

\begin{thebibliography}{00}

\bibitem{B}  K.~S.~Brown,
{\em Cohomology of groups}, Grad. Texts in Math., 
vol.~87, Springer-Verlag, New York-Berlin, 1982.

\bibitem{C} K.~T.~Chen,
{\em Integration in free groups},
Ann. of Math. \textbf{54} (1951), 147--162.

\bibitem{DPS} A.~Dimca, S.~Papadima, A.~Suciu,
{\em Topology and geometry of cohomology jump loci}, 
Duke Math.\ Journal \textbf{148} (2009), no.~3, 405--457.

\bibitem{DP} A.~Dimca, S.~Papadima, 
{\em Arithmetic group symmetry and finiteness properties 
of Torelli groups}, \arxiv{1002.0673}.

\bibitem{DF} W.~G.~Dwyer, D.~Fried,
{\em Homology of free abelian covers. \textup{I}},
Bull. London Math. Soc. \textbf{19} (1987), no.~4, 350--352.

\bibitem{E} D.~Eisenbud,
{\em Commutative algebra with a view towards algebraic geometry},
Grad. Texts in Math., vol.~150, Springer-Verlag, New~York, 1995. 

\bibitem{FH} W.~Fulton, J.~Harris, 
{\em Representation theory}, Grad. Texts in Math., vol.~129,
Springer-Verlag, New York, 1991.

\bibitem{H} R.~Hain,
{\em Infinitesimal presentations of the {T}orelli groups},
J. Amer. Math. Soc. \textbf{10} (1997), no.~3, 597--651.

\bibitem{H1} R.~Hain,
{\em Relative weight filtrations on completions of mapping class
groups}, Groups of diffeomorphisms, 309--368, Adv. Stud. Pure Math., 52,
Math. Soc. Japan, Tokyo, 2008.

\bibitem{Hi} E.~Hironaka,
{\em Alexander stratifications of character varieties},
Ann. Inst. Fourier Grenoble \textbf{47} (1997), no.~2, 555--583.

\bibitem{J0} D.~Johnson,
{\em An abelian quotient of the mapping class group $\TT_g$}, 
Math. Ann. \textbf{249} (1980), 225--242.

\bibitem{J1} D.~Johnson,
{\em The structure of the Torelli group \textup{I}: A finite set 
of generators for $\TT$}, 
Ann. of Math. \textbf{118} (1983), 423--442.

\bibitem{J2} D.~Johnson,
{\em The structure of the Torelli group \textup{II}: A characterization
of the group generated  by twists on bounding curves}, 
Topology \textbf{24} (1985), no.~2, 113--126.

\bibitem{J3} D.~Johnson,
{\em The structure of the Torelli group \textup{III}: The
abelianization of $\TT$}, 
Topology \textbf{24} (1985), no.~2, 127--144.


\bibitem{M} H.~Matsumura,
{\em Commutative algebra}, Second edition,
Math. Lecture Note Ser, vol.~56, Benjamin/Cummings, 
Reading, MA, 1980. 


\bibitem{Mo2} S.~Morita,
{\em The extension of Johnson's homomorphism from the Torelli group to the
mapping class group}, Invent.\ Math.\ \textbf{111} (1993), 197--224. 

\bibitem{PS0} S.~Papadima, A.~Suciu,
{\em Chen {L}ie algebras}, Int. Math. Res. Notices 
\textbf{2004}, no.~21, 1057--1086. 

\bibitem{PS} S.~Papadima, A.~Suciu,
{\em Bieri--{N}eumann--{S}trebel--{R}enz invariants and 
homology jumping loci}, Proc. London Math. Soc. \textbf{100} 
(2010), no.~3, 795--834.

\bibitem{PS1} S.~Papadima, A.~Suciu,
{\em Homological finiteness in the Johnson filtration of  
the automorphism group of a free group}, \arxiv{1011.5292}.

\bibitem{P} A.~Putman,
{\em The Johnson homomorphism and its kernel}, 
\arxiv{0904.0467}.


\bibitem{Q}  D.~Quillen,
{\em Rational homotopy theory}, Ann. of Math. 
\textbf{90} (1969), 205--295.

\bibitem{Se} J.-P.~Serre,
{\em Lie algebras and Lie groups. Lectures given at Harvard University, 1964.}
W. A. Benjamin, Inc., New York-Amsterdam, 1965.

\bibitem{stallings} J.~Stallings,
{\em Homology and central series of groups} J.\ Algebra \textbf{2} (1965),
170--181.

\bibitem{S}  D.~Sullivan,
{\em Infinitesimal computations in topology}, Inst. Hautes 
\'{E}tudes Sci. Publ. Math. \textbf{47} (1977), 269--331.


\end{thebibliography}

\end{document}